\documentclass[oneside,10pt]{amsart}
\usepackage{enumerate, hyperref}
\usepackage{latexsym, bbm, amssymb, amsmath, wasysym, mathtools}
\usepackage[round]{natbib}
\usepackage{setspace}
\usepackage{paralist}
\usepackage{ifpdf}
\usepackage[capposition=top]{floatrow}

\usepackage{tikz}
\usetikzlibrary{patterns,snakes,shapes}

\usepackage{accents}

\makeatletter
\def\subsection{\@startsection{subsection}{1}%
  \z@{.5\linespacing\@plus.7\linespacing}{.3\linespacing}%
  {\normalfont\scshape}}
\newcommand{\eqnum}{\refstepcounter{equation}\textup{\tagform@{\theequation}}}
\makeatother

\theoremstyle{plain}
\newtheorem{theorem}{Theorem}

\newtheorem{lemma}[theorem]{Lemma}
\newtheorem{proposition}[theorem]{Proposition}
\theoremstyle{definition}

\allowdisplaybreaks[1]

\let\oldmarginpar\marginpar
\renewcommand\marginpar[1]{\-\oldmarginpar[\raggedleft\footnotesize #1]%
{\raggedright\footnotesize #1}}

\newcommand{\R}{\mathbb{R}}

\newcommand{\E}{\mathbb{E}}
\newcommand{\F}{\mathcal{F}}

\newcommand{\hhat}{\widehat h}
\newcommand{\vhat}{\widehat v}

\newcommand{\1}{\mathbbm{1}}

\DeclarePairedDelimiter\floor{\lfloor\,}{\,\rfloor}

\begin{document}

\title[]%
{An Adaptive $\boldsymbol{O(\log n)}$-Optimal Policy for the \\Online Selection of a Monotone Subsequence \\from a Random Sample}
\author[]
{Alessandro Arlotto, Yehua Wei, and Xinchang Xie}

\thanks{A. Arlotto:  The Fuqua School of Business, Duke University, 100 Fuqua Drive, Durham, NC, 27708.
Email address: \texttt{alessandro.arlotto@duke.edu}}

\thanks{Y. Wei: Carroll School of Management, Boston College, 140 Commonwealth Avenue, Chestnut Hill, MA 02467.
Email address: \texttt{yehua.wei@bc.edu}}

\thanks{X. Xie: The Fuqua School of Business, Duke University, 100 Fuqua Drive, Durham, NC, 27708.
Email address: \texttt{xinchang.xie@duke.edu}}

\begin{abstract}
        Given a sequence of $n$ independent random variables with common continuous distribution,
        we propose a simple adaptive online policy that selects a monotone increasing subsequence.
        We show that the expected number of monotone increasing selections made by such a policy
        is within $O(\log n)$ of optimal.
        Our construction provides a direct and natural way for proving the $O(\log n)$-optimality
        gap. An earlier proof of the same result made
        crucial use of a key inequality of \citet{BruDel:SPA2001} and
        of de-Poissonization.


        \bigskip


        \bigskip


\end{abstract}

\date{first version: May 12, 2016.
This version: December 18, 2016.
}

\maketitle


\section{Introduction}

In the problem of \emph{online} selection of a \emph{monotone increasing} subsequence,
a decision maker observes sequentially a sequence of independent non-negative random variables $X_1, X_2, \ldots$
with common continuous distribution $F$ and seeks to construct a monotone subsequence
\begin{equation}\label{eq:monotone-subsequence}
X_{\tau_1} \leq  X_{\tau_2} \leq \cdots \leq  X_{\tau_j}
\end{equation}
where the indices $1 \leq \tau_1 < \tau_2 < \cdots < \tau_j$ are stopping times
with respect to the $\sigma$-fields $\F_i = \sigma\{X_1, X_2, \ldots, X_i \}$, $1 \leq i <\infty $,
and the trivial $\sigma$-field $\F_0$.
Since the indices are required to be possible values of stopping times,
all selection/rejection decisions are terminal.
That is, if the decision maker chooses not to select the value $X_{i}$ at time $i$,
then that value is lost forever.
Similarly, if $X_{i}$ is selected at time $i$, then
that selection cannot be changed in the future.
In general, the stopping times can be chosen
to optimize different objective functions,
and two main approaches have been considered in the literature.
In the first, the decision maker seeks to maximize
the expected number of selected elements when $n$ are sequentially revealed
\citep{SamSte:AP1981}.
In contrast, in the second approach the decision maker's objective is to minimize
the expected time it takes to construct a monotone subsequence with $n$ elements
\citep{ArlottoMosselSteele:RSA2016}.
Here, we confine our attention to the first --- more classical --- approach.

We then call a sequence of stopping times $1 \leq \tau_1 < \tau_2 < \cdots < \tau_j \leq n$
such that \eqref{eq:monotone-subsequence} holds
a \emph{feasible policy}, and we denote the set of all such policies
by $\Pi(n)$.
For any $\pi \in \Pi(n)$, we then let $L_n(\pi)$ be the random variable that
counts the number of selections made by policy $\pi$ for the sample $\{X_1,X_2, \ldots, X_n\}$.
That is,
$$
L_n(\pi)=\max \{j: X_{\tau_1} \leq  X_{\tau_2}\leq \cdots \leq  X_{\tau_j} \text{ where } 1\leq \tau_1 < \tau_2 < \cdots < \tau_j \leq n \}.
$$
\citet{SamSte:AP1981} first studied this selection problem and found that for each $n \geq 1$
there is a unique policy $\pi^*_n \in \Pi(n)$ such that
\begin{equation}\label{eq:uniquepolicy}
\E[L_n(\pi^*_n)] = \sup_{\pi \in \Pi(n)} \E[L_n(\pi)],
\end{equation}
and for such optimal policies one has the asymptotic estimate
\begin{equation}\label{eq:SSasymp}
\E[L_n(\pi^*_n)] \sim ( 2n )^{1/2} \quad \text{as } n \rightarrow \infty.
\end{equation}

Over the last few decades the understanding of policy $\pi^*_n$
has substantially evolved.
For instance, by formulating \eqref{eq:uniquepolicy} as a finite-horizon Markov
decision problem one sees that the optimal policy $\pi^*_n$ is characterized
by time and state dependent acceptance intervals.
Furthermore, the asymptotic estimate \eqref{eq:SSasymp} was refined with
the much tighter bounds
\begin{equation}\label{eq:ELn-optimal-bounds}
(2n)^{1/2} - O(\log n) \leq \E[L_n(\pi^*_n)] \leq (2n)^{1/2} \quad \quad \text{ as } n \rightarrow \infty.
\end{equation}

The upper bound in \eqref{eq:ELn-optimal-bounds} was first discovered by \citet{BruRob:AAP1991} while studying
the maximal number of elements of a random sample whose sum is less than a specified value.
The analysis was instigated by the work of \citet{CofFlaWeb:AAP1987},
and it now represents one of the early steps into the domain of
resource-dependent branching processes \citep[see, e.g.,][]{BrussDuerinckx:AOAP2015}.
The result of \citet{BruRob:AAP1991} is actually quite rich;
recent extensions and applications are discussed in \citet{Steele:MA2016}.
The upper bound in \eqref{eq:ELn-optimal-bounds} also appeared in
\citet{Gne:JAP1999} who considered the sequential selection of a monotone increasing subsequence from
a random sample with random size.

The $O(\log n)$ lower bound in \eqref{eq:ELn-optimal-bounds} is much more recent.
It first appeared in the work of \citet{BruDel:SPA2001}
who studied the mean-optimal sequential selection of a monotone increasing
subsequence when the observations $X_1, X_2, \ldots$ are revealed
at the arrival epochs of a unit-rate Poisson process on $[0,n]$.
While the \citet{BruDel:SPA2001} result provides compelling evidence that a similar
bound should also hold for the discrete-time formulation we
consider here --- i.e. the formulation in which the observations are revealed at the times $1, 2, \ldots, n$ --- the sequential nature of the two selection processes
makes the result of \citet{BruDel:SPA2001} not immediately applicable.
The connection between the continuous-time formulation of \citet{BruDel:SPA2001}
and the discrete-time optimization \eqref{eq:uniquepolicy}
was then argued by \citet{ArlNguSte:SPA2015} who
used the concavity of the map $n \mapsto \E[L_n(\pi^*_n)]$
and the $O(\log n)$-bound of \citet{BruDel:SPA2001}
to ultimately confirm the lower bound in \eqref{eq:ELn-optimal-bounds}.

After a careful analysis, \citet{BruDel:SPA2004} proved that the mean-optimal number of monotone increasing selections
with Poisson-many observations is asymptotically normal after
centering around $(2n)^{1/2}$ and scaling by $3^{-1/2}(2n)^{1/4}$.
\citet{ArlNguSte:SPA2015} showed that the same
asymptotic limit also holds for the discrete-time problem with $n$
observations so, in summary, we now know that
\begin{equation}\label{eq:ANS-CLT}
\frac{ 3^{1/2}  \{ L_n(\pi^*_n) - ( 2 n )^{1/2} \} }{ (2 n)^{1/4}} \Longrightarrow N(0,1), \quad \quad \text{as } n \rightarrow \infty.
\end{equation}

However, the analyses of \citet{BruDel:SPA2001,BruDel:SPA2004}
and \citet{ArlNguSte:SPA2015} do not address
whether there is a simple \emph{adaptive} online policy --- i.e., a policy that
depends on the value of the last selection
and on the number of observations that are yet to be seen ---
that is $O(\log n)$ optimal.
The works of \citet{RheTal:JAP1991} and \citet{ArlottoSteele:2011}
tell us that the best non-adaptive policy is $O(n^{1/4})$ optimal,
but this optimality gap is too crude.
For instance, the expected number of monotone increasing selections made by the best non-adaptive policy
cannot even be used to center the random variable $L_n(\pi^*_n)$
around $( 2 n )^{1/2}$ in the weak law \eqref{eq:ANS-CLT}.

In this paper, we construct a \emph{simple adaptive online policy} $\widehat \pi_n$ that is
$O(\log n)$ optimal.
The policy is characterized by a sequence of functions $\hhat_n, \hhat_{n-1}, \ldots, \hhat_1$
such that if the value of the last selection up to and including time $i$ is $s$
and if $k= n-i$ observations remain to be seen
then the value $X_{i+1}$ is selected if and only if $X_{i+1}$ falls in the \emph{acceptance interval} $[s , \hhat_{n-i}(s)]$.
In terms of the stopping times, the policy $\widehat \pi_n$ corresponds to setting
$\widehat \tau_0 = 0$, $X_{\widehat \tau_0}=0$, and then defining the stopping times
$\widehat\tau_1 < \widehat\tau_2 < \cdots < \widehat\tau_j$
recursively as
$$
\widehat\tau_j = \min \{ \widehat\tau_{j-1} < i \leq n: X_i \in [X_{\widehat \tau_{j-1}}, \hhat_{n-i+1}(X_{\widehat \tau_{j-1}})]\}
\quad \text{for } 1 \leq j \leq n,
$$
with the convention that if the set of indices on the right-hand side is empty, then $\widehat \tau_j = \infty$.
The random variable $L_n (\widehat \pi_n)$ then denotes the number of monotone increasing selections
made by policy $\widehat \pi_n$,
and the expected value of $L_n (\widehat \pi_n)$ satisfies the
two bounds given in the next theorem.

\begin{theorem}[$O(\log n)$-Optimal Policy]\label{th:pi-hat-gap}
For each $n\geq 1$, there is a simple adaptive online policy $\widehat \pi_n$ such that
\begin{equation}\label{eq:ELn-hat-bounds}
(2n)^{1/2} - 2\{ \log (n) + 1\} \leq \E[L_n(\widehat \pi_n)] \leq \E[L_n(\pi^*_n)] \leq (2n)^{1/2}.
\end{equation}
\end{theorem}

We discuss the structure of policy $\widehat \pi_n$ and of the functions $\{\hhat_k: 1 \leq k <\infty\}$
in Section \ref{se:pi-hat}, and then we turn to the proof of Theorem \ref{th:pi-hat-gap}.
The upper bound in \eqref{eq:ELn-hat-bounds} immediately follows from \eqref{eq:ELn-optimal-bounds}
and the optimality of policy $\pi^*_n$, but
our analysis requires a generalization of the upper bound \eqref{eq:ELn-optimal-bounds} which we study in Section \ref{se:upper-bound}.
The proof of the lower bound \eqref{eq:ELn-hat-bounds} then follows in Sections \ref{se:rediduals}
and \ref{se:proof-lower-bound-completion}.
Finally, in Section \ref{se:conclusions} we make concluding remarks
and underscore some open problems.

\section{Policy $\widehat \pi_n$ and its Value Function}\label{se:pi-hat}

For any feasible policy $\pi$ and any continuous distribution $F$,
we see that the number of selections made by  $\pi$
for the sample $\{X_1, X_2, \ldots, X_n\}$ is unchanged if
we replace each $X_i$ by its monotone transformation $F^{-1}(X_i)$.
Thus, we can assume without loss of generality that
the $X_i$'s are uniformly distributed on $[0,1]$.
Next, we let
\begin{equation}\label{eq:h-hat}
\hhat_k(s) = \min\left\{  s + [2k^{-1}(1-s)]^{1/2} ,\, 1 \right\} \quad \quad \text{ for all } s \in [0,1] \text{ and all } k \geq 1,
\end{equation}
and we use the sequence of functions $\{\hhat_k: 1 \leq k < \infty\}$ to construct
appropriate acceptance intervals.
Specifically, if $s$ denotes the value of the last selection when $k$
observations are yet to be seen and $x$ is the $k$-to-last presented value,
then $x$ is selected as element of the subsequence that is under construction
if and only if $x \in [s, \hhat_k(s)]$.

We now define the \emph{critical value} $s_k = \max\{1 - 2k^{-1},0\}$ and we note that for
$s \in [0, s_k]$ the decision maker is \emph{conservative} and selects the $k$-to-last value $x$ if and only if it is within
$\{2k^{-1}(1-s)\}^{1/2}$ of the most recent selection $s$.
On the other hand, if $s \in [s_k,1]$ the decision maker is \emph{greedy} and accepts any $k$-to-last value $x$ that
is larger than the most recent selection $s$.

If $k$ denotes the number of observations that are yet to be seen
and $s$ is the value of the most recent selection,
then we let  $\vhat_k(s)$ denote the expected number of monotone increasing selections
made by the acceptance interval policy characterized by the functions $\hhat_k, \hhat_{k-1}, \ldots, \hhat_1$.
The functions $\{\vhat_k : 1 \leq k \leq n\}$ are the value functions associated
with policy $\widehat \pi_n$ and they can be obtained recursively.
Specifically, if $\vhat_0(s) = 0$ for all $s \in [0,1]$, then for $k \geq 1$ we have the recursion
\begin{equation}\label{eq:v-hat}
	\widehat{v}_{k}(s)
    = \{ 1- \widehat{h}_k(s) + s \} \widehat{v}_{k-1}(s)
    + \int_s^{\widehat{h}_k(s)} \{ 1 + \widehat{v}_{k-1}(x) \}  \, dx.
\end{equation}

To see why this recursion holds, we condition on the $k$-to-last
uniform random value. With probability
$1- \widehat{h}_k(s) + s$ the newly presented value $x$ does not fall
in the acceptance interval $[s, \hhat_k(s)]$.
In this case no selection is made and we are left with $k-1$ remaining observations
and with the value of the most recent selection $s$ unchanged. This amounts to
an expected number of remaining selections equal to $\vhat_{k-1}(s)$ and
it justifies the first summand of our recursion \eqref{eq:v-hat}.
On the other hand, with probability $\widehat{h}_k(s) - s$ the newly presented
value $x$ falls in the acceptance interval, and we obtain a reward of one
for selecting $x$ plus the expected number of remaining selections over the next $k-1$ observations
when the value of the most recent selection changes to $x$.
Integrating this over all $x \in [s, \hhat_k(s)]$ gives us
the second summand of the recursive equation \eqref{eq:v-hat}.
The value functions $\{\vhat_k : 1 \leq k < \infty \}$  are all continuous on $[0,1]$ and
their behavior is well summarized by the following theorem.

\begin{theorem}[$\vhat_k$ Bounds]\label{th:lower-bound}
For all $k \geq 1$ and all $s \in [0,1]$ one has that
\begin{equation}\label{eq:main-bounds}
\{2k(1-s)\}^{1/2} - 2 \{ \log( k ) + 1 \} \leq \vhat_k (s) \leq \{2k(1-s)\}^{1/2}.
\end{equation}
\end{theorem}

Since policy $\widehat \pi_n$ is characterized by the thresholds $\hhat_n, \hhat_{n-1}, \ldots, \hhat_1$
and by the initial state $s=0$, one then has the equivalence
$$
\E[L_n(\widehat \pi_n)] = \vhat_n(0) \quad \quad \text{for all } n \geq 1.
$$
Hence, Theorem \ref{th:pi-hat-gap} is an immediate corollary of Theorem \ref{th:lower-bound},
but the upper bound \eqref{eq:main-bounds} is a refinement of \eqref{eq:ELn-optimal-bounds}
to an arbitrary initial state.
In fact the same upper bound holds for all feasible policies based on acceptance intervals, including
the optimal one.

The estimates in Theorem \ref{th:lower-bound}
also allow us to provide some intuition for our
choice of the threshold functions $\{ \hhat_k: 1 \leq k < \infty\}$
in \eqref{eq:h-hat}.
For every $k \geq 1$ and $s \in [0,1]$ the
threshold functions aim to balance the expected reward to-go
$\vhat_{k-1}(s)$ that one obtains when skipping the $k$-to-last
observation $x$, and the reward $1 + \vhat_{k-1}(x)$ that one
earns when selecting the $k$-to-last value $x$.
Since $\vhat_{k-1}(s) \approx \{2(k-1)(1-s)\}^{1/2}$,
one can solve the equation
\begin{equation}\label{eq:approximate-equation}
  \{2(k-1)(1-s)\}^{1/2} = 1 + \{2(k-1)(1-\widehat{x})\}^{1/2}
\end{equation}
to find the largest value of $x$ that makes selecting the current
value worthwhile.
Equation \eqref{eq:approximate-equation} then tells us that
\begin{equation*}
 \widehat{x} = s + [2(k-1)^{-1}(1-s)]^{1/2} - [2(k-1)]^{-1}
             \approx s + [2k^{-1}(1-s)]^{1/2} - [2k]^{-1},
\end{equation*}
so our choice \eqref{eq:h-hat} accounts for the first two terms of the
approximation of the solution of equation \eqref{eq:approximate-equation}.
The truncation in \eqref{eq:h-hat} then ensures that all the thresholds
$\{ \hhat_k: 1 \leq k < \infty\}$ are feasible.
At the end of the next section we use an optimization
argument to provide further intuition for
our choice of the threshold functions.

\section{A Refined Prophet Upper Bound}\label{se:upper-bound}

In this section, we prove the upper bound \eqref{eq:main-bounds} by showing that it holds for all policies
that are based on acceptance intervals.
The adaptive policy $\widehat \pi_n$ and the unique optimal policy $\pi^*_n$
both have this property.
The argument we provide draws substantially from the earlier analyses of \citet{Gne:JAP1999} and \citet{BruDel:SPA2001},
but it takes advantage of the flexibility that comes from allowing for an arbitrary initial state.
Specifically, here we assume that the first subsequence element can be selected only if it is larger than
an arbitrary value $s \in [0,1]$. In contrast, in the classical formulation one always takes the initial state $s=0$.

For any $k \geq 1$, an arbitrary acceptance interval policy $\pi_k$
is given by a sequence $h_k, h_{k-1}, \ldots, h_1$ of functions such that
$$
s \leq h_j(s) \leq 1 \quad \quad \text{for all } 1 \leq j \leq k \text{ and all } s \in [0,1].
$$
If $s$ denotes the initial state or the value of the last observation selected prior to being presented with
the $j$-to-last value $x$, then $x$ is selected if and only if $x \in [s, h_j(s)]$.
Next, for any $s \in [0,1]$ we set $M_0 = s$ and we let
$M_i$ denote the maximum between $M_0$ and the largest of the elements of the subsequence
that have been selected up to and including time $i$.
The number of selections from $\{X_1, X_2, \ldots, X_k\}$
when $M_0=s$  is then given by
the random variable
$$
L_k(\pi_k,s) = \sum_{i=1}^k \1(X_i \in [M_{i-1}, h_{k-i+1}(M_{i-1})]).
$$
If we now take expectations on both sides and
use the Cauchy-Schwarz inequality
to estimate the sum of the products $1 \cdot \E[ h_{k-i+1}(M_{i-1}) - M_{i-1} ]$
for $1 \leq i \leq k$, we obtain
\begin{align}\label{eq:CS-first-application}
\E[ L_k(\pi_k,s)]&=\sum_{i=1}^k 1 \cdot \E[ h_{k-i+1}(M_{i-1}) - M_{i-1} ] \\
&\leq k^{1/2} \bigg\{ \sum_{i=1}^k  \E[ h_{k-i+1}(M_{i-1}) - M_{i-1} ]^2 \bigg\}^{1/2}.\notag
\end{align}

The definition of $M_i$ as the maximum between $M_0$ and the largest
of the elements of the subsequence that have been selected up to and including time $i$ tells us that
\begin{equation*}
M_i =
\begin{cases}
M_{i-1} \quad &\text{if } X_i \notin [M_{i-1}, h_{k-i+1}(M_{i-1}) ]\\
X_i \quad &\text{if } X_i \in [M_{i-1}, h_{k-i+1}(M_{i-1})]
\end{cases}
\end{equation*}
and, because $X_i$ is uniformly distributed on the unit interval, we have the identity
$$
\E[M_i -M_{i-1}\, | \, \F_{i-1} ] = \int_{M_{i-1}}^{h_{k-i+1}(M_{i-1})} (x -  M_{i-1}) \, dx
= \frac{1}{2}\left( h_{k-i+1}(M_{i-1})- M_{i-1}\right)^2.
$$
Taking the total expectation then gives
\begin{equation*}
\E[\left( h_{k-i+1}(M_{i-1}) - M_{i-1}\right)^2] = 2 \{ \E[M_i] - \E[M_{i-1}] \},
\end{equation*}
so a second application of the Cauchy-Schwarz inequality implies the upper bound
$$
\E[ h_{k-i+1}(M_{i-1}) - M_{i-1}]^2 \leq  2 \{ \E[M_i] - \E[M_{i-1}] \}.
$$
If we now sum over $1 \leq i \leq k$ and recall that $\E[M_k]\leq 1$ and $M_0 = s$,
we obtain from telescoping that
\begin{equation}\label{eq:sum-exp-increment-squared-bound}
\sum_{i=1}^{k} \E[ h_{k-i+1}(M_{i-1}) - M_{i-1}]^2 \leq 2 \{ \E[M_k] - \E[M_{0}] \} \leq 2 (1-s).
\end{equation}
This last inequality and the bound in \eqref{eq:CS-first-application} finally
give us that
$$
\E[ L_k(\pi_k,s)] \leq \{2k(1-s)\}^{1/2} \quad \quad \text{for all } k \geq 1 \text{ and all } s \in [0,1]
$$
and complete the proof of the upper bound \eqref{eq:main-bounds}.

The same upper bound can also be obtained by formulating
a simple optimization problem that provides further insight into
our choice of the adaptive thresholds $\{\widehat{h}_k: 1 \leq k < \infty\}$.
We consider the optimization problem
\begin{align}\label{eq:optimization-problem}
w^* =  \max_{d_1, \ldots, d_k} & \sum_{i=1}^{k} d_i\\
            \text{s.t.}\,\,\, & \sum_{i=1}^{k} d_i^2 \leq 2(1-s),\notag
\end{align}
and we obtain from inequality \eqref{eq:sum-exp-increment-squared-bound} that $d_i = \E[h_{k-i+1}(M_{i-1}) - M_{i-1}]$ is a feasible solution.
This feasible solution has objective
$$
\sum_{i=1}^k \E[ h_{k-i+1}(M_{i-1}) - M_{i-1} ] = \E[ L_k(\pi_k,s)],
$$
so we have that the optimal objective value of \eqref{eq:optimization-problem}
is an upper bound for $\E[ L_k(\pi_k,s)]$, and that \eqref{eq:optimization-problem}
is a relaxation of the online monotone subsequence problem \eqref{eq:uniquepolicy}.

The optimal value of \eqref{eq:optimization-problem} can be easily estimated.
The problem has a linear objective function and convex constraints;
by the Karush--Kuhn--Tucker (KKT) conditions,
its optimal solution $(d_1^*, d_2^*, \ldots, d_k^*)$ is given by
$$
d_i^* = [2 k^{-1}(1-s)]^{1/2} \quad \quad \text{for all } 1 \leq i \leq k,
$$
and its optimal value is
$$
w^* = \{2 k (1-s)\}^{1/2} \quad \quad \text{for all } k \geq 1 \text{ and } s \in[0,1].
$$

It then follows that the adaptive thresholds
$\{\widehat{h}_k: 1 \leq k < \infty\}$ defined by \eqref{eq:h-hat}
are \emph{reoptimized} thresholds that use the optimal solution of the relaxation \eqref{eq:optimization-problem} for any given $k \geq 1$ and $s \in[0,1]$.
Of course, a small nuisance arises, and one needs to make sure
that the reoptimized thresholds are indeed feasible.
Since there are values $s \in [0,1]$ such that $s + [2 k^{-1}(1-s)]^{1/2}  > 1$,
one needs the truncation introduced in \eqref{eq:h-hat}
and separately consider the conservative and greedy regimes.

\section{Residual Functions and Residual Differences}\label{se:rediduals}

The upper bound \eqref{eq:main-bounds} tells us that $\widehat{v}_k(s) \leq \{2k(1-s)\}^{1/2}$
for all  $s \in [0,1]$ and all $ 1 \leq k < \infty$,
so if we consider the \emph{residual function}
\begin{equation}\label{eq:residual}
  r_k(s) = \{2k(1-s)\}^{1/2} - \widehat{v}_k(s)
  \quad \quad \text{for } s \in [0,1] \text{ and } 1 \leq k < \infty,
\end{equation}
we can obtain a recursion for $\widehat{v}_k$ that is equivalent to \eqref{eq:v-hat}.
Specifically, when we substitute the value function $\widehat v_{k-1}(y)$ with the difference
 $[2(k-1)(1-y)]^{1/2}  - r_{k-1}(y)$ for all $y \in [0,1]$ on the right-hand side of the recursive equation \eqref{eq:v-hat}
and rearrange, we obtain that
\begin{align}\label{eq:v-hat-second-recursion}
\widehat{v}_{k}(s)  = &  [2(k - 1)(1 - s)]^{1/2} \\
                      & - \int_s^{\widehat{h}_k(s)}  \{ [2(k - 1)(1 - s)]^{1/2} - [2(k - 1)(1 - x)]^{1/2} - 1 \}  \, dx \notag\\
                      & -  \{ 1- \widehat{h}_k(s) + s \} r_{k-1}(s) - \int_s^{\widehat{h}_k(s)}  r_{k-1}(x)   \, dx. \notag
\end{align}
Next, we define the function $\delta_k : [0,1] \rightarrow \R$ as
\begin{align}\label{eq:delta-k}
 \delta_{k}(s)=  & [2k(1-s)]^{1/2} - [2(k-1)(1-s)]^{1/2} \\
         & + \int_s^{\widehat{h}_k(s)} \{ [2(k-1)(1-s)]^{1/2} - [2(k-1)(1-x)]^{1/2} - 1\} \, dx, \notag
\end{align}
so when we plug-in the recursion \eqref{eq:v-hat-second-recursion} on the right-hand side of \eqref{eq:residual}
and use the definition of $\delta_k(s)$, we see that the residual function  \eqref{eq:residual} can be written as
\begin{equation}\label{eq:residual-alternative}
r_k(s) = \delta_{k}(s) + \{ 1- \widehat{h}_k(s) + s \} r_{k-1}(s) + \int_s^{\widehat{h}_k(s)}  r_{k-1}(x)   \, dx.
\end{equation}

For each $k \geq 1$, the residual function $r_k(s)$ is continuous and defined on a compact interval,
so if we maximize with respect to $s$ we find the \emph{maximal residual}
\begin{equation}\label{eq:max-residual}
  r_k =  \max_{0 \leq s \leq 1 } r_k(s)
  \quad \quad \text{for }  1 \leq k < \infty.
\end{equation}
If we now substitute the residual functions on the right-hand side of \eqref{eq:residual-alternative}
with their maximal value, we find the inequality
\begin{equation}\label{eq:residual-function-bound}
r_k(s) \leq \delta_{k}(s) + r_{k-1}  \quad \quad \text{for all } s \in [0,1] \text{ and all } 1 \leq k < \infty.
\end{equation}
This inequality implies that the residual difference $r_k(s) - r_{k-1}$ is bounded
above by $\delta_{k}(s)$, and we will shortly see that the behavior of the function $\delta_k$
changes in the conservative and greedy regimes.
Nevertheless, the function $\delta_k$ is appropriately bounded for each $k \geq 1$.

\begin{proposition}[Residual Differences Bound]\label{pr:delta-k-bound}
For each $k \geq 1$ one has that
\begin{equation}\label{eq:delta-k-bound}
r_{k}(s) - r_{k-1}  \leq \delta_k(s) \leq \frac{2}{k} \quad \quad \text{for all } s \in [0,1].
\end{equation}
\end{proposition}

The first inequality of \eqref{eq:delta-k-bound} is just a rearrangement of \eqref{eq:residual-function-bound},
so we focus on the second inequality.
Its proof is given in the next two
sections.
Section \ref{se:monotonicity-delta-k-conservative} studies the conservative regime, while
Section \ref{se:monotonicity-delta-k-greedy} deals with the greedy regime.

\subsection{The Function $\delta_k$ in the Conservative Regime}\label{se:monotonicity-delta-k-conservative}

We begin our analysis of the function $\delta_k$ in the conservative regime by stating
the following monotonicity lemma.

\begin{lemma}[Monotonicity of the $\delta$-Functions in the Conservative Regime]\label{lm:monotonicity-delta-k}
For each $k \geq 3$ and $s_k =  1 - 2k^{-1}$, the function $\delta_k: [0,1] \rightarrow \R$ defined by \eqref{eq:delta-k}
is non-decreasing on $[0, s_k]$. Moreover, we have that
\begin{equation}\label{eq:delta-k-bound-conservative}
\delta_k(s) \leq \delta_k(s_k) \leq \frac{4}{3k}  \quad \quad \text{for all } s \in [0, s_k].
\end{equation}
\end{lemma}

We prove the monotonicity of $\delta_k$ on $[0, s_k]$ by
showing that the derivative $\delta'_k$ is non-negative on $(0,s_k)$.
For $k \geq 3$ and $s \in [0,s_k]$,
we see from \eqref{eq:h-hat} that the function
$\hhat_k(s) = s + [2k^{-1}(1-s)]^{1/2}$,
and we obtain from \eqref{eq:delta-k} that $\delta_k$
is continuous and differentiable on $(0,s_k)$.
Furthermore, we also have that $0 \leq [k(1-s)]^{-1} \leq 2^{-1}$,
so when we differentiate both sides of \eqref{eq:delta-k}
and rearrange, we obtain that
\begin{equation}\label{eq:delta-k-derivative}
\delta'_k(s) = \left[\frac{(k - 1)(1 - s)}{2}\right]^{1/2} \gamma_k \left( \frac{1}{k(1-s)} \right),
\end{equation}
where $\gamma_k$ is defined for all $y \in [0, 2^{-1}]$ by
$$
\gamma_k(y) = 2
                + y k \left[ 1 - (1 - k^{-1})^{1/2} \right]
                - 2 [ 1 - (2 y )^{1/2}]^{1/2}
                -  ( 2 y )^{1/2}  \left\{  2 - [ 1 -  ( 2 y )^{1/2}]^{1/2}  \right\}.
$$

Of course, here we have that the first factor of \eqref{eq:delta-k-derivative} is non-negative
because $k \geq 1$ and $s \in [0,1]$,
and we prove the non-negativity of the second factor by showing that
$\gamma_k$ is non-negative on the whole interval $[0,2^{-1}]$.
As first step we recall the inequality
$$
\frac{1}{2} =  k \left[ 1 - \bigg( 1 - \frac{1}{2k} \bigg)\right] \leq  k \left[ 1 - (1 - k^{-1})^{1/2} \right],
$$
and we replace the second summand of $\gamma_k$ with this last lower bound.
This then implies that for all $y \in [0, 2^{-1}]$, one has that
$$
 2 + \frac{1}{2}y - 2 [ 1 - (2 y )^{1/2}]^{1/2} -  ( 2 y )^{1/2} \left\{ 2 - [ 1 -  ( 2 y )^{1/2}]^{1/2} \right\}
 \leq \gamma_k(y).
$$
The next lemma confirms that the left-hand side of this last inequality
is non-negative, completing the proof of the monotonicity property in Lemma \ref{lm:monotonicity-delta-k}.

\begin{lemma}\label{lm:infinite-series-inequality}
If $y \in [0, 2^{-1}]$, then one has the inequality
\begin{equation}\label{eq:infinite-series-inequality}
 (2y)^{1/2} \left\{ 2 - \big[1 - (2y)^{1/2} \big]^{1/2} \right\} \leq 2 + \frac{1}{2}y - 2 \big[ 1 - (2y)^{1/2} \big]^{1/2}.
\end{equation}
\end{lemma}

The proof of Lemma \ref{lm:infinite-series-inequality} uses a
infinite series representation of the two quantities that appear
on each side of inequality \eqref{eq:infinite-series-inequality}.
The next lemma represents and intermediate step,
and it isolates an important property of the coefficients
that appear in such series.

\begin{lemma}[Coefficient Difference Inequality]\label{lm:coefficients-difference}
If
$$
a_j = \frac{ (2j)! }{ (2j-1) (j!)^2 2^{3j/2}},
$$
then one has the inequality
$$
2^{1/2} a_j - a_{j-1} \geq 0 \quad \quad \text{for all } j \geq 3.
$$
\end{lemma}

\begin{proof}
By definition of $a_j$ we have $a_j \geq 0$ for all $j \geq 1$.
Moreover, we also have the recursion
$$
a_j = \left(\frac{2j-3}{2^{1/2} j}\right) a_{j-1}.
$$
It then follows that the difference
$$
2^{1/2} a_j - a_{j-1} = \left(\frac{2j-3}{j} \right) a_{j-1} - a_{j-1} = \left(\frac{j-3}{j}\right) a_{j-1} \geq 0 \quad \text{for all } j \geq 3,
$$
just as needed.
\end{proof}

With the inequality of Lemma \ref{lm:coefficients-difference} at our disposal,
we now carry out the proof of Lemma \ref{lm:infinite-series-inequality}.

\begin{proof}[Proof of Lemma \ref{lm:infinite-series-inequality}]
For $j \geq 1$, we let
$$
a_j = \frac{ (2j)! }{ (2j-1) (j!)^2 2^{3j/2}},
$$
and we note that for $y \in [0,2^{-1}]$ we have the infinite series representation
\begin{equation}\label{eq:infinite-series-representation}
- \big[ 1 - (2y)^{1/2} \big]^{1/2} = -1 + \sum_{j=1}^\infty a_j y^{j/2}.
\end{equation}

The infinite series representation \eqref{eq:infinite-series-representation} and the
fact that $a_1 = 2^{-1/2}$ implies that the left-hand side of \eqref{eq:infinite-series-inequality},
$ \ell = (2y)^{1/2}\left\{ 2 - \big[ 1 - (2y)^{1/2} \big]^{1/2} \right\}$,
can be written as
$$
\ell = (2y)^{1/2} +  y +  \sum_{j=2}^\infty 2^{1/2} a_j y^{(j+1)/2}.
$$
Similarly, the infinite series representation \eqref{eq:infinite-series-representation}
and the estimates $a_1 = 2^{-1/2}$ and $a_2 = 4^{-1}$
tell us that the right-hand side of \eqref{eq:infinite-series-inequality}$, r = 2 + 2^{-1} y - 2 \big[ 1 - (2y)^{1/2} \big]^{1/2}$,
can be expressed as
$$
r =  (2y)^{1/2} + y  + \sum_{j=3}^\infty 2 a_j y^{j/2}.
$$
In turn, one obtains the difference
$$
r-\ell = \sum_{j=3}^\infty 2 a_j y^{j/2} - \sum_{j=2}^\infty 2^{1/2} a_j y^{(j+1)/2} = 2^{1/2}  \sum_{j=3}^\infty \{2^{1/2} a_j - a_{j-1}\} y^{j/2},
$$
and one has that
$
r - \ell \geq 0
$
because of the non-negativity of the differences $\{ 2^{1/2} a_j - a_{j-1}: j \geq 3\}$ established in Lemma \ref{lm:coefficients-difference}.
\end{proof}

The upper bound \eqref{eq:delta-k-bound-conservative} then follows after one notices
that
$$
\delta_k(s_k) =  2 \bigg[ 1 - \left(1-\frac{1}{k}\right)^{1/2} - \frac{1}{k} \bigg] + \frac{4}{3k} \left(\frac{k-1}{ k } \right)^{1/2}.
$$
For any $y \in [0,1]$ we know that $1- (1-y)^{1/2} \leq y$ so we see that the first summand on the right-hand side
is non-positive while the second summand is simply bounded by $4(3k)^{-1}$.
This then completes the proof of the upper bound in Lemma \ref{lm:monotonicity-delta-k}.

\subsection{The Function $\delta_k$ in the Greedy Regime}\label{se:monotonicity-delta-k-greedy}

If $ s \in [s_k, 1]$ for each $k \geq 1$ then the decision maker selects
any feasible observation and the function $\hhat_k(s) = 1$.
Thus, we see from carrying the integration in \eqref{eq:delta-k} that $\delta_k$
can be written as
\begin{equation}\label{eq:delta-k-greedy}
\delta_k(s) = (1-s)^{1/2} \big\{ (2k)^{1/2} - [2(k-1)]^{1/2} \big\} + (1-s)\bigg\{ \frac{1}{3} [2(k-1)(1-s)]^{1/2} -1 \bigg\}.
\end{equation}
We now recall that $s_k = \max\{0, 1-2 k^{-1}\}$ and that for $s \in [s_k, 1]$ one has that
$$
0 \leq 1-s \leq 2k^{-1}.
$$
If we now replace the two factors $(1-s)^{1/2}$ in \eqref{eq:delta-k-greedy}
with their upper bound $(2k^{-1})^{1/2}$ and rearrange, we obtain the inequality
$$
\delta_k(s) \leq 2 \bigg[ 1 - \left( 1 - \frac{1}{k}\right)^{1/2} \bigg] + (1-s)\bigg[ \frac{2}{3}\left(\frac{k-1}{k}\right)^{1/2} -1 \bigg].
$$
Again, for any $y \in [0,1]$ one has that $1- (1-y)^{1/2} \leq y$ so the first summand on the right-hand side is bounded by $2 k^{-1}$.
Moreover, we also see that the second summand is non-positive
and we finally obtain that
$$
\delta_k(s) \leq \frac{2}{k} \quad \quad \text{for all } s \in [s_k,1].
$$
This last bound and inequality \eqref{eq:delta-k-bound-conservative} together
complete the proof of \eqref{eq:delta-k-bound}.

\section{Completion of the Proof of the Lower Bound}\label{se:proof-lower-bound-completion}

Inequality \eqref{eq:delta-k-bound} in Proposition \ref{pr:delta-k-bound}
plays a crucial role in the proof of the lower bound \eqref{eq:main-bounds}.
In fact, it suffices to recall from the definition of the maximal residual \eqref{eq:max-residual}
that
$$
r_k = \max_{0 \leq s \leq 1} \{ [2k(1-s)]^{1/2} - \vhat_k(s) \}
$$
and to write $r_k$ as the telescoping sum
$$
r_k = \sum_{j=1}^k \{r_j - r_{j-1}\}.
$$
The bound \eqref{eq:delta-k-bound} tells us that each summand on the right-hand
side is bounded by $2j^{-1}$ so we obtain that
$$
r_k \leq 2\{ \log(k) + 1\} \quad \quad \text{for all } k \geq 1,
$$
completing the proof of Theorem \ref{th:lower-bound}.

\section{Connections and Observations}\label{se:conclusions}

Policy $\widehat \pi_n$ is a simple adaptive online policy that selects a
monotone increasing subsequence.
It is then easy to see how it could be generalized to the
sequential selection of a unimodal or $d$-modal subsequence considered by \citet{ArlottoSteele:2011}.
For instance, in the unimodal case with $n$ observations,
one could set the turning time $\bar n = \floor{ n/2 }$ and
run policy $\widehat \pi_{\bar n}$ to construct an increasing
segment with the first $\bar n$ observations,
and a decreasing version of $\widehat \pi_{n-\bar n}$
to obtain a decreasing segment over the next $n -\bar n$ observations.
Theorem \ref{th:pi-hat-gap} then would immediately apply to this
construction, proving a $O(\log n)$-optimality gap for the
sequential selection of unimodal and $d$-modal subsequences.

It is also reasonable to expect that policy $\widehat \pi_n$ could generalize to
the sequential selection of coordinatewise increasing subsequences
from a uniform random sample on the $m$-dimensional hypercube.
Thus far, the best optimality-gap estimate comes from the work of \citet{BarGne:AAP2000}
who use a non-adaptive policy to derive a $O(n^{1/(2m+2)})$ bound.
The adaptive character of policy $\widehat \pi_n$ could be fruitful one more time
and help establish a $O(\log n)$-optimality gap for this multidimensional problem.

Many interesting questions remain open, however.
Theorem \ref{th:lower-bound} tells us that
$$
\{2k(1-s)\}^{1/2} - \vhat_k (s) \leq 2 \{ \log( k ) + 1 \}
\quad \quad \text{ for all } s\in [0,1] \text{ and all } k \geq 1,
$$
and it would be worthwhile to understand how the right-hand side changes
with the initial state value $s$.
When $s = 0$ we have from Theorem \ref{th:pi-hat-gap} that
$$
\E[ L_n (\pi^*_n) ] \leq \E[ L_n (\widehat \pi_n) ] + 2\{\log(n) +1\} \quad \quad\text{for all } n \geq 1,
$$
but the actual expected performance of policy $\widehat \pi_n$ seems
to be much tighter.
Based on an extensive numerical analysis\footnote{We estimated numerically the value functions that solve the recursion \eqref{eq:v-hat} and
its optimal counterpart on a discretized state space with a grid size of $10^{-5}$
and with $k$ ranging from $1$ to $10^{4}$.},
we conjecture that there is a constant $0 < c < \infty$ such that
$$
\E[ L_n (\pi^*_n) ] \leq \E[ L_n (\widehat \pi_n) ] + c \quad \quad\text{for all } n \geq 1.
$$
Of course, this would be a substantial improvement and it would
also imply that the functions $\widehat h_n, \widehat h_{n-1}, \ldots, \widehat h_1$
are remarkably close to their analogues that arise when implementing the
optimal dynamic programming algorithm.

Furthermore, it would also be interesting to study in greater detail the difference
$$
g(n) = (2n)^{1/2} - \E[ L_n (\pi^*_n) ].
$$
In the closely related problem in which the sequence $X_1, X_2, \ldots, X_n$
is given by a uniform random permutation of the integers $\{1,2,\ldots,n\}$,
\citet{PeichaoSteele:PAMS2016} showed that the corresponding difference is $O(\log n)$
as $n \rightarrow \infty$.
In our context with independent and identically distributed observations,
the simplest open question is whether $g(n)$ diverges as $n$ grows to infinity,
but its answer is unlikely to be easy.

\section*{Acknowledgement}\label{se:conclusions}

We are grateful to Sa$\check{\rm s}$a Peke$\check{\rm c}$ and J. Michael Steele for
their thoughtful comments on an earlier draft of this manuscript.
This material is based upon work supported by the National Science Foundation
under Grant No. 1553274.

\end{document}